\newcommand{\set}[1]{\left\{#1\right\}}

\newcommand{\norm}[1]{\lVert#1\rVert}
\newcommand{\trinorm}[1]{{\left\vert\kern-0.25ex\left\vert\kern-0.25ex\left\vert #1 
    \right\vert\kern-0.25ex\right\vert\kern-0.25ex\right\vert}}

\newcommand{\jump}[1]{\llbracket #1 \rrbracket}
\newcommand{\mean}[1]{\{#1\}}

\newcommand{\n}{\boldsymbol{n}}
\newcommand{\tg}{\boldsymbol{t}}
\newcommand{\x}{\boldsymbol{x}}

\newcommand{\bu}{\boldsymbol{u}}
\newcommand{\bv}{\boldsymbol{v}}
\newcommand{\bj}{\boldsymbol{j}}
\newcommand{\bH}{\boldsymbol{h}}

\newcommand{\bE}{\boldsymbol{e}}

\renewcommand{\div}{\mathrm{div}}
\newcommand{\curl}{\mathbf{curl}}

\newcommand{\bPi}{\boldsymbol{\Pi}}

\newcommand{\cT}{\mathcal{T}}
\newcommand{\cF}{\mathcal{F}}
\newcommand{\cE}{\mathcal{E}}

\newcommand{\cP}{\mathcal{P}}

\newcommand{\rmH}{\mathrm{H}}
\newcommand{\rmL}{\mathrm{L}}

\newcommand{\C}{\mathbb C}

\documentclass[12pt]{article}

\usepackage{amsmath,amsthm,amssymb,amsfonts}
\usepackage{stmaryrd}
\usepackage{latexsym,graphicx,color}

\newtheorem{lemma}{Lemma}[section]
\newtheorem{theorem}{Theorem}[section]
\newtheorem{prop}{Proposition}[section]

\date{}
\title{A discontinuous Galerkin method for the time harmonic eddy current problem
\thanks{Support by the University of Trento and by
the Spanish Ministry of Economy Project MTM2013-43671-P.}}

\author{{\sc Ana Alonso Rodr\'iguez}\thanks{
Department of Mathematics, University of Trento,
Trento, Italy,
e-mail: {\tt ana.alonso@unitn.it}}, $\,\,$
{\sc Salim Meddahi}\thanks{Departamento de Matem\'aticas, Facultad de Ciencias,
Universidad de Oviedo, Calvo Sotelo s/n, Oviedo, Espa\~na,
e-mail: {\tt salim@uniovi.es}}\\ and \\
{\sc Alberto Valli}\thanks{Department of Mathematics, University of Trento,
	Trento, Italy, e-mail: {\tt
alberto.valli@unitn.it}}
}

\begin{document}

\maketitle

\begin{abstract}
We introduce and analyze a discontinuous Galerkin method for a time-harmonic eddy current problem formulated in terms of the magnetic field. The scheme is obtained by putting together a DG method for the approximation of the vector field variable representing the magnetic field in the conductor and a DG method for the Laplace equation whose solution is a scalar magnetic potential in the insulator. The transmission conditions linking the two problems are taken into account weakly in the global discontinuous Galerkin scheme. We prove that the numerical method is uniformly stable and obtain  quasi-optimal error estimates in the DG-energy norm.
\end{abstract}

\section{Introduction}
In this paper, we present  a discontinuous Galerkin (DG) approximation of a time-harmonic eddy current problem. The eddy current approximation of Maxwell equations is obtained by disregarding the displacement current term. It is commonly used in applications related with induction heating, transformers, magnetic levitation and non-destructive testing. These problems often involve composite materials and structures, complex transmission conditions and, eventually, boundary  layers due to the skin effect. The ability of DG methods to handle efficiently unstructured meshes with hanging nodes combined with $hp$-adaptive strategies make them well-suited for the numerical simulation of physical systems related to eddy currents.  

The eddy current problem is generally written in terms of  either the electric or the magnetic field, cf. \cite{AVbook}. These two formulations are equivalent at the continuous level but they lead to  different numerical schemes. A discontinuous Galerkin method based on a time-harmonic eddy current problem written in terms of the electric field has been analyzed in the pioneering work of Perugia and Schotzau \cite{PS03}. For the time domain eddy current problem, Ausserhofer et al. introduced in \cite{ABP09} a formulation based on a magnetic vector potential and propose a numerical method that combines a DG approximation in the conductor with the usual $\rmH^{1}$-conforming Lagrange finite element approximation in the insulator. 

Here, we are interested in imposing the magnetic field as primary unknown. The advantage of this approach rests on the reduction of the number of degrees of freedom resulting from the introduction of a scalar magnetic  potential in the  nonconducting medium. The global formulation of the problem consists in a $\rmH(\curl)$-elliptic problem for vector fields that are curl-free in the insulator $\Omega_I$. Our DG formulation is obtained by applying for the Laplace equation posed in $\Omega_{I}$ the usual interior penalty finite element method, that can be traced back to \cite{arnoldIP}, see also \cite{DiPietroErn} and the references cited therein for more details. In the conductor $\Omega_{C}$ we employ, as in \cite{HPSS05, PS03}, the interior penalty method corresponding to the N\'ed\'elec curl-conforming finite element space of the second kind.  We point out that the introduction of discrete harmonic fields is necessary when considering domains of general topology. We prove the stability of the resulting combined DG scheme by exploiting the elliptic character of the problem. We also obtain, under adequate regularity assumptions, quasi-optimal asymptotic error estimates. It is worthwhile to notice that the implementation of the DG-method presented here only requires the use of standard shape functions.  The  curl-conforming finite elements, more precisely, the N\'ed\'elec finite elements of the second kind, are only needed for the theoretical convergence results in Section 5. 


The outline of this paper is as follows. In Section 2 we derive the model problem used in the finite element approximation. We introduce our  DG  formulation in Section 3. Finally, Section 4 is devoted to the convergence analysis, and asymptotic error estimates are provided in Section 5.

\section{The model problem}\label{s2}

Let $\Omega_C\subset \mathbb{R}^3$ be a bounded polyhedral domain with a Lipschitz boundary $\Gamma$.
We denote by $\mathbf{n}_\Gamma$ the unit normal vector on $\Gamma$ that points towards
$\Omega_e:= \mathbb{R}^3\setminus \overline \Omega_C$. In order to illustrate the impact of the conductor's topology in our method, we assume that $\Omega_C$ has a toroidal shape. We notice that the eddy current problem is posed in the whole space with asymptotic conditions on the behaviour of the electric and magnetic fields at infinity. Depending on the nature of the eddy current problem being solved and the geometry involved, a discretization method can be obtained for this problem by either applying a pure finite element  approach on a truncated domain or by using a combination of boundary (BEM) and finite elements (FEM), see \cite{AMV,Hiptmair,MS03,AV09}.  The FEM-BEM formulation is posed in the conductor but its implementation is more difficult and it leads to more complex algebraic linear systems of equations. The FEM method 
needs a large computational domain, but it is simpler and it can provide an alternative in many practical situations. It is the option that we will consider in the following. To this end, we introduce a bounded domain $D$ containing in its interior $\overline{\Omega}_{C}$ and whose connected boundary $\Sigma= \partial D$ is located at a large enough distance from the conductor $\Omega_{C}$. The bounded domain $\Omega_I:= D\setminus \overline \Omega_C$ represents then the nonconducting region of the computational domain $D$. 

Under our assumptions, the first de Rham cohomology group $\mathcal H^1(\Omega_I)$ of $\Omega_{I}$, namely, the space of curl-free vector fields that are not gradients, has dimension one.  If we assume that $\Omega_I$ is a polyhedral domain endowed with a tetrahedral mesh, one can use the technique given in \cite{BRS02} for the explicit construction of a piecewise-linear vector field  $\boldsymbol \rho$ spanning $\mathcal H^1(\Omega_I)$ and satisfying $\boldsymbol \rho \times {\bf n}_\Sigma = {\bf 0}$ on $\Sigma$, where ${\bf n}_\Sigma$ denotes the outward unit normal vector to $\Sigma$. For an alternative construction of  $\boldsymbol \rho$ see Alonso Rodr\'{\i}guez et al. \cite{ABGV13}.

The eddy current problem formulated in terms of the magnetic field $\bH$  and the scalar magnetic potential $\psi$  reads as follows:
\begin{equation}\label{model0}
\begin{array}{rcll}
\imath \omega \mu \bH + \curl\, \bE  &=& \boldsymbol 0         &\text{in $D$}\\[2ex]
 \bE  &=& \sigma^{-1} ( \curl\, \bH - \bj )         &\text{in $\Omega_C$}\\[2ex]
        \bH &=& \nabla \psi + k \boldsymbol \rho  &\text{in $\Omega_I$}\\[2ex]
        \psi &=& 0  & \text{on } \Sigma  \, ,
\end{array}
\end{equation}
where $\bj$ is the applied current density, $\mu$ is the magnetic permeability and $\sigma$ 
is the electric conductivity. In what follows, we assume that $\mu$ and $\sigma$ are positive piecewise constant functions in $\Omega_{C}$ and that $\mu_{|\Omega_I}= \mu_0$ is the permeability constant of vacuum. It follows from the first equation \eqref{model0} that  
\begin{equation}\label{lap}
0=\div(\bH_{|\Omega_I})= \div(\nabla \psi + k \boldsymbol \rho)\, \hbox{ in } \Omega_I\, .
\end{equation}
We point out here that the electric field $\bE$ is not uniquely determined in $\Omega_{I}$. 
Nevertheless,  the tangential components of the magnetic field and the tangential components of 
any admissible representation of the electric field should be continuous across 
the interface $\Gamma$, i.e.,  
\begin{equation}\label{magt}
\bH|_{\Omega_{C}} \times \n_\Gamma = (\nabla \psi + k \boldsymbol \rho) \times \n_\Gamma \, .
\end{equation}
and
\begin{equation}\label{mage}
\bE_{|\Omega_C} \times \n_\Gamma = \bE_{|\Omega_I} \times \n_\Gamma.
\end{equation}
The electric field $\bE$  is considered here as an auxiliary variable, it will be removed from the formulation. Hence, we should deduce from \eqref{mage} a transmission condition relating $\bH$ and $\varphi$ on $\Gamma$. Applying the surface divergence operator 
$\div_\Gamma$ to both side of \eqref{mage} and recalling that $\div_\Gamma(\bE \times \n_\Gamma)= \curl \,\bE \cdot \n_\Gamma$ we deduce that the field $\curl \,\bE$ admits continuous normal components across $\Gamma$. As a consequence of the  first equation of \eqref{model0}, $\mu\bH$ should also have continuous normal components across $\Gamma$, i.e.,   
\begin{equation}\label{transe}
\mu \bH \cdot \n_\Gamma = \mu_{0}(\nabla \psi + k \boldsymbol \rho) \cdot \n_\Gamma \,. 
\end{equation}
Finally, we deduce from \eqref{mage} and the property $\curl \, \boldsymbol \rho = {\bf 0}$  that 
$$
\int_\Gamma  \bE_{|\Omega_C}  \times \n_\Gamma \cdot \boldsymbol \rho =\int_\Gamma  \bE_{|\Omega_I}  \times \n_\Gamma \cdot \boldsymbol \rho = \int_{\Omega_I} \curl \, \bE \cdot \boldsymbol \rho,
$$
 thus
\begin{equation}\label{scal}
\int_\Gamma \sigma^{-1} ( \curl\, \bH - \bj )  \cdot (\boldsymbol \rho \times \n_\Gamma)= \imath \, \omega \int_{\Omega_I} \mu_0  (\nabla \psi + k \boldsymbol \rho) \cdot \boldsymbol \rho \, .
\end{equation}

From now on, for the sake of simplicity in notations, $\bH$ will stand for $\bH|_{\Omega_{C}}$. 
Taking into account \eqref{lap}, \eqref{magt}, \eqref{transe} and \eqref{scal}, we deduce that 
the eddy current problem can be formulated in terms of the magnetic field and its scalar potential representation in the insulator in the following form: Find $\bH:\, \Omega_{C}\to \mathbb C^{3}$, $\psi:\, \Omega_{I}\to \mathbb C$ and $k\in \mathbb C$ such that,
\begin{align}
\imath \omega \mu \bH + \curl\, [\sigma^{-1} ( \curl\, \bH - \bj )] &= \boldsymbol 0          &\text{in $\Omega_C$}\label{ModelProblem1}\\[2ex]
\bH \times \n_\Gamma & =   (\nabla \psi + k \boldsymbol \rho) \times \n_\Gamma  &\text{on $\Gamma$}\label{ModelProblem2}\\[2ex]
\mu\, \bH \cdot \n_\Gamma & =   \mu_0(\nabla \psi + k \boldsymbol \rho) \cdot \n_\Gamma &\text{on $\Gamma$}\label{ModelProblem3}\\[2ex]
\int_\Gamma \sigma^{-1} ( \curl\, \bH - \bj )  \cdot (\boldsymbol \rho \times \n_\Gamma)&=\imath \,  \omega \mu_0 \int_{\Omega_I}  (\nabla \psi + k \boldsymbol \rho) \cdot \boldsymbol \rho\label{ModelProblem4}\\[2ex]
\div(\nabla \psi + k \boldsymbol \rho) & =  0 &\text{in $\Omega_I$}\label{ModelProblem5}\\[2ex]
\psi &= 0   &\text{on } \Sigma\label{ModelProblem6}  \, .
\end{align}

We refer to \cite[Section 5]{AVbook} for a proof of the well-posedness of problem \eqref{ModelProblem1}-\eqref{ModelProblem6}.

\section{The discrete problem}\label{section3}

\subsection{Notations}
Given a real number $r\geq 0$ and a polyhedron $\mathcal O\subset \mathbb R^d$, $(d=2,3)$,
we denote the norms and seminorms of the usual Sobolev space
$\rmH^r(\mathcal O)$ by $\|\cdot \|_{r,\mathcal O}$ and $|\cdot|_{r,\mathcal O}$ respectively
(cf. \cite{McLean}). We use the convention $\rmL^2(\mathcal O):= \rmH^0(\mathcal O)$ and ${\bf L}^2(\mathcal O):= [\rmL^2(\mathcal O)]^3$.
We recall that, for any $t \in [-1,\: 1 ]$, the spaces $\rmH^{t}(\partial \mathcal O)$
have an intrinsic definition (by localization) on the Lipschitz surface $\partial \mathcal O$
due to their invariance under Lipschitz coordinate transformations. Moreover, for all $0< t\leq 1$,
$\rmH^{-t}(\partial\mathcal O)$ is the dual of $\rmH^{t}(\partial\mathcal O)$ with respect
to the pivot space $\rmL^2(\partial \mathcal{O})$. Finally we consider $\mathbf{H}(\curl, \mathcal O):=\{ \bv \in \rmL^2(\mathcal O)^3 \, : \, \curl \, \bv \in \rmL^2(\mathcal O)^3\}$ and endow it with its usual Hilbertian norm $\norm{\bv}_{\mathbf{H}(\curl, \mathcal O)}^2:=
\norm{\bv}_{0, \mathcal O}^2 + \norm{\curl \, \bv}_{0, \mathcal O}^2$.

We consider a sequence  $\{\cT_h\}_h$ of conforming and shape-regular triangulations of 
$\overline \Omega_C \cup \overline \Omega_I$. 
We assume that each partition $\cT_h$ consists of tetrahedra $K$ of diameter $h_K$ and unit outward normal to $\partial K$ denoted $\n_K$. We also assume that for all $K\in \cT_h$ we have either $K\subset \overline\Omega_C$ or $K\subset \overline \Omega_I$ and denote
\[
\cT_h^{\Omega_C}:= \set{K\in \cT_h;\quad K\subset \overline\Omega_C},\qquad \cT_h^{\Omega_I}:= \set{K\in \cT_h;\quad K\subset \overline\Omega_I}.
\]
We also assume that the meshes $\{\cT_h^{\Omega_C}\}_h$ are aligned with the discontinuities of the coefficients $\sigma$ 
and $\mu$. The parameter $h:= \max_{K\in \cT_h} \{h_K\}$ represents the mesh size.

We denote by $\cF_h^0(\Omega_C)$ and $\cF_h^0(\Omega_I)$ the sets of interior faces of the triangulations 
$\cT_h^{\Omega_C}$ and $\cT_h^{\Omega_I}$ respectively. 
We also introduce the sets of boundary faces
\[
\cF_h^\Gamma:= \set{F = \overline K\cap\overline{K'};\quad  K\in \cT_h^{\Omega_C},\,\, K'\in \cT_h^{\Omega_I}}
\quad 
\text{and}
\quad
\cF_h^\Sigma:= \set{F = \partial K \cap \Sigma;\quad K\in \cT_h^{\Omega_I}}
\]
and consider
\[
\cF_h^{\Omega_C}:=\cF_h^0(\Omega_C) \cup  \cF_h^\Gamma, \quad \cF_h^{\Omega_I}:=\cF_h^0(\Omega_I) \cup  \cF_h^\Sigma
\quad \text{and} \quad \cF_h := \cF_h^{\Omega_C}\cup \cF_h^{\Omega_I}.
\]

We notice that $\set{\cF_h^\Gamma}_h$ is a shape regular family of triangulations of $\Gamma$ into triangles $T$ of diameter $h_T$. Finally, we consider the set $\cE_h$ of edges $e = \overline T\cap \overline{T'}$ (where  $T$ and  $T'$ are two adjacent triangles from  $\cF_h^\Gamma$). 

Let $\mathcal{O}_h$ be anyone of the previously introduced partitions of $\overline\Omega_C\cup \overline \Omega_I$, $\overline\Omega_C$, $\overline \Omega_I$ or $\Gamma$ and let $E$ be a generic element of the given partition. We introduce for any  $s\geq 0$ the broken Sobolev spaces
\[
 \rmH^s(\mathcal{O}_h) := \prod_{E\in \mathcal{O}_h} \rmH^s(E)\quad \text{and} \quad  \mathbf{H}^s(\mathcal{O}_h) := \prod_{E\in \mathcal{O}_h} \mathbf{H}^s(E)^3\, .
 \]

For each $w:= \set{w_E}\in \rmH^s(\mathcal{O}_h)$, the components  $w_E$ 
represents the restriction $w|_E$. When no confusion arises, the restrictions will be written
without any subscript.

The space 
$\rmH^s(\mathcal{O}_h)$ is endowed with the Hilbertian norm
\[
\norm{w}_{s,\mathcal{O}_h}^2 := \sum_{E\in \mathcal{O}_h} \norm{w_E}^2_{s,E}.
\]

We consider identical definitions for the norm and the seminorm on the vectorial version $\mathbf{H}^s(\mathcal{O}_h)$.
We use the standard conventions $\rmL^2(\mathcal{O}_h):=\rmH^0(\mathcal{O}_h)$ and 
$\mathbf{L}^2(\mathcal{O}_h):=\mathbf{H}^0(\mathcal{O}_h)$ and introduce the  bilinear forms
\[
(w, z)_{\mathcal{O}_h} = \sum_{E\in \mathcal{O}_h} \int_E w_E z_E, \quad \forall w, z \in \rmL^2(\mathcal{O}_h)
\]
and
\[
(\boldsymbol{w}, \boldsymbol{z})_{\mathcal{O}_h} = \sum_{E\in \mathcal{O}_h} \int_E \boldsymbol{w}_E\cdot \boldsymbol{z}_E, \quad \forall \boldsymbol{w}, \boldsymbol{z}\in \mathbf{L}^2(\mathcal{O}_h).
\]


Assume that  $(\bv,\varphi,m)\in \mathbf{H}^{1+s}(\cT_h^{\Omega_C})\times \rmH^{1+s}(\cT_h^{\Omega_I})\times \C$, with $s>1/2$.
Moreover, let us recall that $\boldsymbol \rho$ has been constructed as a piecewise-linear vector field, therefore its restriction to any face $F$ has a meaning.
We define $\curl_h\bv\in \mathbf{H}^s(\cT_h^{\Omega_C})$ by $(\curl_h \bv)|_K = \curl \, \bv_K$, for all $K\in \cT_h^{\Omega_C}$; $\nabla_h \varphi \in \mathbf{H}^s(\cT_h^{\Omega_I})$ by $(\nabla_h \varphi)|_K = \nabla \varphi_K$,  for all $K\in \cT_h^{\Omega_I}$. 
We define also the averages $\mean{\bv}_{\cF}\in \mathbf{L}^2(\cF_h^{\Omega_C})$ and
$\mean{\nabla_h \varphi+m\boldsymbol \rho}_{\cF}\in \mathbf{L}^2(\cF_h^{\Omega_I})$ by 
\begin{equation} \label{average1}
\begin{array}{l}
\mean{\bv}_{\cF}|_F := \mean{\bv}_F  \hbox{ with}\\ \\
\mean{\bv}_F:= \left\{
\begin{array}{ll}
(\bv_K + \bv_{K'})/2 & \text{if $F=K\cap K'\in \cF_h^0(\Omega_C)$}\\[.1cm]
\bv_K & \text{if $F\subset \partial K$ and $F \in \cF_h^\Gamma$},
\end{array} \right.
\end{array}
\end{equation}
and
\begin{equation} \label{average2}
\begin{array}{l}
 \mean{\nabla_h\varphi+m \boldsymbol \rho}_{\cF}|_F := \mean{\nabla_h \varphi + m \boldsymbol \rho}_F \hbox{ with} \\ \\
 \mean{\nabla_h \varphi + m \boldsymbol \rho}_F :=\left\{ \begin{array}{l} 
(\nabla \varphi_K + \nabla \varphi_{K'})/2 + m (\boldsymbol \rho_K+\boldsymbol \rho_{K'})/2 \\ \hspace{3cm} \text{if $F=K\cap K'\in \cF_h^0(\Omega_I)$}\\[.1cm]
\nabla \varphi_K + m \boldsymbol \rho_K \\
\hspace{3cm} \text{if $F\subset \partial K$ and $F \in \cF_h^\Sigma$}\, ,
\end{array} \right.
\end{array}
\end{equation}
and the jumps $\jump{(\bv,\varphi,m)}_{\cF}\in \mathbf{L}^2(\cF_h^{\Omega_C})$ and $\jump{\varphi\n}_{\cF}\in \mathbf{L}^2(\cF_h^{\Omega_I})$ by
\begin{equation} \label{jump1}
\begin{array}{l}
\jump{(\bv,\varphi,m)}_{\cF}|_F :=\jump{(\bv,\varphi,m)}_F \hbox{ with}
\\ \\
\jump{(\bv,\varphi,m)}_F:=\left\{   \hspace{-.2cm}  \begin{array}{l}
\jump{\bv\times \n}_F:=\bv_K \times \n_K + \bv_{K'}\times \n_{K'} \\ \hspace{2cm} \text{if $F=K\cap K'\in \cF_h^0(\Omega_C)$}\\[.1cm]
\bv_K \times \n +( \nabla \varphi_{K'} +  m \boldsymbol \rho_{K'}) \times \n_{K'}\\ \hspace{2cm} \text{if }F= K\cap K'\in \cF_h^\Gamma \text{ with } K\in \cT_h^{\Omega_C}, \,K'\in \cT_h^{\Omega_I}\, ,
\end{array} \right.
\end{array}
\end{equation}
and
\begin{equation} \label{jump2}
\begin{array}{l}
\jump{\varphi\n}_{\cF}|_F := \jump{\varphi\n}_F \hbox{ with} \\ \\
\jump{\varphi\n}_F := \left\{
\begin{array}{ll}
\varphi_K \n_K + \varphi_{K'}\n_{K'} & \text{if $F=K\cap K'\in \cF_h^0(\Omega_I)$}\\[.1cm]
\varphi_K \n_\Sigma & \text{if $F\subset \partial K$ and $F \in \cF_h^\Sigma$}\, .
\end{array} \right.
\end{array}
\end{equation}
Similarly, we define the edge averages $\mean{\bv}_{\cE}\in \mathbf{L}^2(\cE_h)$ by 
\[
\mean{\bv}_{\cE}|_e :=\mean{\bv}_e \hbox{ with } \mean{\bv}_e:=(\bv_{K_e} + \bv_{K'_e})/2
\]
where $K_e, K'_e\in \cT_h^{\Omega_C}$ are such that $T=\partial K_e \cap \Gamma\in \cF_h^\Gamma$, 
$T'=\partial K'_e \cap \Gamma\in \cF_h^\Gamma$
and $e = T\cap T'$. We also need to define the edge jumps $\jump{\varphi\tg}_{\cE}\in \mathbf{L}^2(\cE_h)$ by
\[
\jump{\varphi\tg}_{\cE}|_e :=   \jump{\varphi\tg}_e \hbox{ with }
 \jump{\varphi\tg}_e := 
 \varphi_{K_e} \tg_e + \varphi_{K'_e} \tg'_e\, ,
\]
where $K_e, K_e'$ are in this case the elements from $\cT_h^{\Omega_I}$ such that $T=\partial K_e \cap \Gamma\in \cF_h^\Gamma$, 
$T'=\partial K'_e \cap \Gamma\in \cF_h^\Gamma$ and $e = T\cap T'$. Here, $\tg_e$, $\tg'_e$ are the tangent unit vectors along the edge $e$ given by
$\tg_e = (\n_\Gamma \times \boldsymbol \nu_{T})|_e$ and  $\tg_e = (\n_\Gamma \times \boldsymbol \nu_{T'})|_e$ where $\boldsymbol \nu_{T}$ and $\boldsymbol \nu_{T'}$ are
the   outward unit  normal vector to $\partial T$ and $\partial T'$ respectively that lies on the tangent plane to $\Gamma$.

\subsection{The DG formulation}

Hereafter, given an integer $k\geq 0$ and a domain
$\mathcal O\subset \mathbb{R}^3$, $\cP_k(\mathcal O)$ denotes the space of polynomials of degree at most $k$ on $\mathcal O$.
For any $m\geq 1$, we introduce the finite element spaces
\[
 \mathbf{X}_h := \prod_{K\in \cT_h^{\Omega_C}} P_m(K)^3
 \quad \text{and} \quad
 V_h := \prod_{K\in \cT_h^{\Omega_I}} \tilde\cP_{m}(K),
\]
where
\begin{equation}\label{tildePm}
 \tilde\cP_{m}(K):=\begin{cases}
 \cP_m(K) & \text{if $\partial K\cap \Gamma \notin \cF_h^\Gamma$ }\\
 \cP_m(K) + \cP_{m+1}^T(K) & \text{if $T=\partial K\cap \Gamma \in \cF_h^\Gamma$}
 \end{cases}
\end{equation}
with $\cP_{m+1}^T(K)$ representing the subspace of $\cP_{m+1}(K)$ spanned by the elements of the Lagrange basis corresponding to nodal points located on $T$. It follows that $\cP_m(K) \subset \tilde\cP_m(K) \subset 
\cP_{m+1}(K)$ and if $T = \partial K \cap \Gamma \in \cF_h^\Gamma$ then $\tilde\cP_m(K)|_T = \cP_{m+1}(T)$.
\vspace{0.25cm}

 
Let $h_\cF\in \prod_{F\in \mathcal{F}_h} \cP_0(F)$ and $h_\cE\in \prod_{e\in \mathcal{E}_h} \cP_0(e)$ be defined by 
$h_\cF|_F := h_F$ $, \forall F \in \mathcal{F}_h$ and $h_\cE|_e := h_e$ $, \forall e \in \cE_h$ respectively. 
By virtue of our hypotheses on $\sigma$ and on the triangulation $\cT_h^{\Omega_C}$, we may consider that 
$\sigma$ is an element of $\prod_{K\in \mathcal{T}_h^{\Omega_C}} \cP_0(K)$ and denote 
$\sigma_K:= \sigma|_K$ for all $K\in \mathcal{T}_h^{\Omega_C}$.
We introduce $\mathtt{s}_\cF\in \prod_{F\in \mathcal{F}_h(\Omega_C)} \cP_0(F)$ defined by 
$\mathtt{s}_F := \min(\sigma_K, \sigma_{K'})$, if $F = \partial K \cap \partial K'\in \cF_h^0(\Omega_C)$ and  
$\mathtt{s}_F := \sigma_K$,  if $F = \partial K \cap \Gamma\in \cF_h^\Gamma$. We also need to define  
$\mathtt{s}_\cE\in \prod_{e\in \mathcal{E}_h} \cP_0(e)$ given by $\mathtt{s}_e = \min(\sigma_{K_e}, \sigma_{K'_e})$
where $K_e, K_e\in \cT_h^{\Omega_C}$ are such that $T=\partial K_e \cap \Gamma\in \cF_h^\Gamma$, 
$T'=\partial K'_e \cap \Gamma\in \cF_h^\Gamma$
and $e = T\cap T'$.
\vspace{0.25cm}

We consider, for $s>1/2$, the Hilbert space 
\[
\mathbf{X}^s(\cT_h^{\Omega_C}) := \left\{\bv \in \mathbf{H}^s(\cT_h^{\Omega_C});\quad \curl_h \bv \in \mathbf{H}^{1/2+s}(\cT_h^{\Omega_C})\right\}
\]
and define on $\mathbf{X}^s(\cT_h^{\Omega_C}) \times \rmH^{1+s}(\cT_h^{\Omega_I}) \times \C$ the sesquilinear forms
\begin{align*}
A_h^{\Omega_C}((\bu, \phi,c), &(\bv, \varphi,m)) :=  \imath \omega \left(\mu \bu, \bv\right)_{\cT_h^{\Omega_C}} + \left(\sigma^{-1} \curl_h \bu, \curl_h \bv\right)_{\cT_h^{\Omega_C}} \\ &+ \left(\mean{\sigma^{-1} \curl_h \bu}_{\cF}, \jump{(\bv,\varphi,m)}_{\cF}\right)_{\cF_h^{\Omega_C}}
+ \left(\mean{\sigma^{-1} \curl_h \bv}_{\cF}, \jump{(\bu, \phi, c)}_{\cF}\right)_{\cF_h^{\Omega_C}}\\ &+ \mathtt{a}^{\Omega_C} \left( \mathtt{s}^{-1}_{\cF} h_{\cF}^{-1} \jump{(\bu, \phi, c)}_{\cF}, \jump{(\bv, \varphi, m)}_{\cF}\right)_{\cF_h^{\Omega_C}}\, ,
\end{align*}
\begin{align*}
A_h^{\Omega_I}(&(\bu, \phi, c), (\bv, \varphi, m)):=  \imath \omega \mu_0 (\nabla_h \phi+c \boldsymbol \rho, \nabla_h \varphi+m \boldsymbol \rho)_{\cT_h^{\Omega_I}} +  \dfrac{\mathtt{a}^{\Omega_I}}{\omega \mu_0}\left(h_{\cF}^{-1} \jump{\phi\n}_{\cF}, \jump{\varphi\n}_{\cF}\right)_{\cF_h^{\Omega_I}}\\
&- \imath \omega \mu_0\left(\mean{\nabla_h \phi+c \boldsymbol \rho}_{\cF}, \jump{\varphi\n}_{\cF}\right)_{\cF_h^{\Omega_I}} 
- \imath \omega \mu_0\left(\mean{\nabla_h \varphi+m \boldsymbol \rho}_{\cF}, \jump{\phi\n}_{\cF}\right)_{\cF_h^{\Omega_I}}
\\ 
&- \left(\mean{\sigma^{-1} \curl_h \bu}_{\cE}, \jump{\varphi\tg}_{\cE}\right)_{\cE_h} - \left(\mean{\sigma^{-1} \curl_h \bv}_{\cE}, \jump{\phi\tg}_{\cE}\right)_{\cE_h}
+ \alpha \left( \mathtt{s}^{-1}_{\cE} h_{\cE}^{-2} \jump{\phi\tg}_{\cE}, \jump{\varphi\tg}_{\cE}  \right)_{\cE_h},
\end{align*}
and let
\[
A_h((\bu, p,c), (\bv, \varphi,m)):= A_h^{\Omega_C}((\bu, \phi, c), (\bv, \varphi, m)) + A_h^{\Omega_I}((\bu, \phi, c), (\bv, \varphi, m))\, .
\]

Let us assume that $\sigma^{-1} \bj\in \mathbf{H}^{1/2+s}(\cT_h^{\Omega_C})$ with $s>1/2$. Then we can define the linear form $L_h(\cdot)$ on $\mathbf{X}^s(\cT_h^{\Omega_C}) \times \rmH^{1+s}(\cT_h^{\Omega_I}) \times \C$ by
\[
L_h((\bv, \varphi)) := (\sigma^{-1} \bj, \curl_h \bv)_{\cT_h^{\Omega_C}} + \left( \mean{\sigma^{-1} \bj}_{\cF}, \jump{(\bv, \varphi, m)}_{\cF} \right)_{\cF_h^{\Omega_C}}
- \left(\mean{\sigma^{-1} \bj}_{\cE}, \jump{\varphi \tg}_{\cE}\right)_{\cE_h}.
\]

We propose the following DG formulation of problem \eqref{ModelProblem1}-\eqref{ModelProblem6}: 
\begin{equation}\label{DG-FEM}
\begin{array}{l}
\text{Find $(\bH_{h}, \psi_h,k_h)\in \mathbf{X}_h\times V_h\times \C$ such that,}\\[2ex]
A_h((\bH_{h}, \psi_h,k_h), (\bv, \varphi, m)) = L_h((\bv, \varphi, m))\quad \forall \,  (\bv, \varphi, m)\in \mathbf{X}_h\times V_h\times \C\, .
\end{array}
\end{equation}

The existence and uniqueness of the solution of this problem is proved in Theorem~\ref{LM}

We end this section by showing that the DG scheme \eqref{DG-FEM} is consistent. 
\begin{prop}\label{consistency0}
	Let $(\bH, \psi, k)\in \mathbf{H}(\curl, {\Omega_C})\times \rmH^1(\Omega_I)\times \C$ be the solution of \eqref{ModelProblem1}-\eqref{ModelProblem6}. Under the assumption $\sigma^{-1} \bj\in \mathbf{H}^{1/2+s}(\cT_h^{\Omega_C})$ and 
	the regularity conditions $(\bH,\psi, k) \in \mathbf{X}^s(\cT_h^{\Omega_C})\times \rmH^{1+s}(\cT_h^{\Omega_I})\times \C$, with $s>1/2$, we have that
	\[
	A_h((\bH, \psi, k), (\bv, \varphi, m)) = L_h((\bv, \varphi, m)) \quad
	\forall \, (\bv,\varphi, m)\in \mathbf{X}_h\times V_h \times \C.
	\]
\end{prop}
\begin{proof}
Using again the notation $\bE  = \sigma^{-1} ( \curl\, \bH - \bj )$ and taking into account that 
$\jump{(\bH,\psi,k)}_{\cF} =0$, $\jump{\psi\n}_{\cF}=0$, and $\jump{\psi\tg}_{\cE}=0$, it is straightforward to show that 
\begin{multline}\label{diff}
A_h((\bH, \psi, k), (\bv, \varphi, m)) - L_h((\bv, \varphi, m)) = \imath \omega \int_{\Omega_{C}} \mu \bH\cdot \bv + \int_{\Omega_{C}} \bE \cdot \curl_h \bv\\
+ \imath \omega \mu_0 \int_{\Omega_I}(\nabla \psi + k\boldsymbol \rho)\cdot (\nabla_h \varphi + m \boldsymbol \rho)
+ \left(\mean{\bE}_{\cF}, \jump{(\bv, \varphi, m)}_{\cF}\right)_{\cF_h^{\Omega_{C}}}\\ - \imath \omega \mu_0\left(\mean{\nabla \psi + k \boldsymbol \rho}_{\cF}, \jump{\varphi\n}_{\cF}\right)_{\cF_h^{\Omega_I}}
- \left(\mean{\bE}_{\cE}, \jump{\varphi\tg}_{\cE}\right)_{\cE_h}.
\end{multline}
Integrating by parts  in each $K\in \cT_h^{\Omega_{C}}$ and using \eqref{ModelProblem1} yield 
\begin{multline}\label{GreenOmegaC}
\int_{\Omega_{C}} \bE \cdot \curl_h \bv = \sum_{K\in \cT_h^{\Omega_{C}}} 
\int_K \curl\, \bE \cdot \bv -
\sum_{K\in \cT_h^{\Omega_{C}}} \int_{\partial K} \bE \cdot \bv\times \n_K \\
= -\imath\omega \int_{\Omega_{C}} \mu\bH\cdot \bv - \sum_{F\in \cF^0_h(\Omega_{C})}
\int_{F} \mean{\bE}_F \cdot \jump{\bv\times \n}_F -
\sum_{T\in \cF_h^\Gamma}
\int_{T} \bE \cdot \bv\times \n.
\end{multline}
Similarly, integrating by parts in each $K\in \cT_h^{\Omega_I}$ together with 
\eqref{ModelProblem4} and \eqref{ModelProblem5} give 
\begin{multline}\label{GreenOmegaI}
\imath \omega \mu_{0}\int_{\Omega_I}(\nabla \psi + k\boldsymbol \rho)\cdot (\nabla_h \varphi + m \boldsymbol \rho) = - \imath \omega \mu_{0}\sum_{K\in \cT_h^{\Omega_I}} \int_K \div (\nabla \psi + k \boldsymbol \rho) \varphi \\+\imath \omega \mu_{0}
\sum_{K\in \cT_h^{\Omega_I}} \int_{\partial K} (\nabla \psi + k\boldsymbol \rho) \cdot  \n_K \varphi
+ m \int_{\Omega_{I}} (\nabla \psi + k\boldsymbol \rho)\cdot \boldsymbol \rho
= \imath \omega \mu_{0} \sum_{F\in \cF_{h}^0} \int_{F} \mean{\nabla \psi + k \boldsymbol \rho}_{F}
\cdot \jump{\varphi\n}_{F}\\
- \imath \omega \mu_{0} \sum_{T\in \cF^{\Gamma}_{h}}\int_{F} (\nabla \psi + k \boldsymbol \rho)
\cdot \varphi\n_{\Gamma} + \imath \omega \mu_{0} \sum_{T\in \cF^{\Sigma}_{h}}\int_{F} (\nabla \psi + k \boldsymbol \rho) \cdot \varphi\n_{\Sigma} + m \int_{\Gamma} \bE \cdot (\boldsymbol \rho \times \n_{\Gamma}).
\end{multline}
Substituting back \eqref{GreenOmegaC} and \eqref{GreenOmegaI} in \eqref{diff} we obtain  
\begin{multline}\label{diff1}
A_h((\bH, \psi, k), (\bv, \varphi, m)) - L_h((\bv, \varphi, m)) = -
\sum_{T\in \cF_h^\Gamma} \int_T \bE \cdot \curl_T \varphi\\ 
- \imath \omega \mu_0 \sum_{T\in \cF_h^\Gamma} \int_T \nabla (\psi+ k \boldsymbol \rho) \cdot \varphi \n_{\Gamma} - \left(\mean{\bE}_{\cE}, \jump{\varphi\tg}_{\cE}\right)_{\cE_h}.
\end{multline}
Finally, using the integration by parts formula 
\begin{equation*}
\sum_{T\in \cF_h^\Gamma}
\int_{T} \bE \cdot \curl_T \varphi =\sum_{T\in \cF_h^\Gamma}
\int_{T} (\text{curl}_T \bE)  \varphi - \sum_{T\in \cF_h^\Gamma} \int_{\partial T} \bE \cdot \varphi\tg_{\partial T}
=
\int_{\Gamma} (\text{curl}_\Gamma \bE)  \varphi -
\left(\mean{\bE}_{\cE}, \jump{\varphi\tg}_{\cE}\right)_{\cE_h}, 
\end{equation*}
we deduce from \eqref{diff1} that 
\begin{multline*}\label{diff1+}
A_h((\bH, \psi, k), (\bv, \varphi, m)) - L_h((\bv, \varphi, m)) = -
\int_{\Gamma} (\text{curl}_\Gamma \bE)  \varphi\\ 
- \imath \omega \mu_0 \sum_{T\in \cF_h^\Gamma} \int_T \nabla (\psi+ k \boldsymbol \rho) \cdot \varphi \n_{\Gamma}.
\end{multline*}
and the result follows from the identity $\text{curl}_\Gamma \bE = \curl \bE \cdot \n$, equation \eqref{ModelProblem1} and the transmission condition \eqref{ModelProblem3}.
\end{proof}

\section{Convergence analysis of the DG-FEM formulation}\label{section4}

The aim of this Section is to prove that the DG-FEM formulation \eqref{DG-FEM} is stable in the 
DG-norm defined on $\mathbf{X}^s(\cT_h^{\Omega_C})\times \rmH^{1+s}(\cT_h^{\Omega_I})\times \C$ by \begin{align*}
\norm{(\bv, \varphi,m)}^2 := & \norm{(\omega\mu)^{1/2}\bv}^2_{0,\Omega_C} + \norm{\sigma^{-1/2}\curl_h \bv}^2_{0,\Omega_C} + \omega\mu_0 \norm{\nabla_h \varphi+ m \boldsymbol \rho}^2_{0,\Omega_I} \\
+ & \norm{\mathtt{s}^{-1/2}_{\cF}h_{\cF}^{-1/2}\jump{(\bv,\varphi,m) }_{\cF}}^2_{0,\cF_h^{\Omega_C}} + \omega\mu_0 \norm{h_{\cF}^{-1/2}\jump{\varphi\n}_{\cF}}^2_{0,\cF_h^{\Omega_I}}\\
+ &\norm{\mathtt{s}^{-1/2}_{\cE} h_{\cE}^{-1}\jump{\varphi\tg}_{\cE}}^2_{0,\cE_h}.
\end{align*}
We also need to introduce 
\begin{multline*}
\norm{(\bv, \varphi,m)}_{\ast}^2 :=  \norm{(\bv, \varphi,m)}^2 + 
\norm{\mathtt{s}_{\cF}^{1/2} h_{\cF}^{1/2} \mean{\sigma^{-1}\curl_h \bv}_{\cF}}^2_{0,\cF_h^{\Omega_C}}\\
+ \norm{\mathtt{s}_{\cE}^{1/2} h_{\cE} \mean{\sigma^{-1} \curl_h \bv }_{\cE}}^2_{0,\cE_h}
+\norm{h_{\cF}^{1/2} \mean{\nabla_h \varphi + m \boldsymbol \rho}_{\cF}}^2_{0,\cF_h^{\Omega_I}}.
\end{multline*}

The following discrete trace inequality is  standard, (see, e.g. \cite[Lemma 1.46]{DiPietroErn}).
\begin{lemma}
 For all integer $k\geq 0$ there exists a constant $C^*>0$ independent of $h$ such that,
 \begin{equation}\label{discreteTrace3D} 
  h_Q \norm{ v}^2_{0,\partial Q} \leq C^* \norm{ v}^2_{0,Q} \quad \forall \, v\in \cP_k(Q),\quad 
  \forall Q\in \{\cT_h,\cF_h^\Gamma\}.
 \end{equation}
\end{lemma}
It is used to prove the following auxiliary result. 
\begin{lemma} \label{equivalence}
For all $k\geq 0$, there exist constants $C_{\Omega_C}>0$ and $C_{\Omega_I}>0$  independent 
of the mesh size and the coefficients such that 
\begin{equation}\label{discIneq1}
\norm{\mathtt{s}_{\cE}^{1/2} h_{\cE} \mean{\sigma^{-1} \mathbf{w}}_{\cE}} _{0,\cE_h} +
\norm{\mathtt{s}_{\cF}^{1/2} h_{\cF}^{1/2} \mean{\sigma^{-1}\mathbf{w}}_{\cF}}_{0,\cF_h^{\Omega_C}} \leq C_{\Omega_C} \norm{\sigma^{-1/2} \mathbf{w}}_{0,{\Omega_C}}\, ,
\end{equation}
 for all $\mathbf{w} \in \prod_{K\in \cT_h^{\Omega_C}}\cP_k(K)^3$,
and
\begin{equation}\label{discIneq2}
\norm{h_{\cF}^{1/2} \mean{\mathbf{w}}_{\cF}}_{0,\cF_h^{\Omega_I}} \leq C_{\Omega_I} \norm{\mathbf{w}}_{0,\Omega_I} \, ,
\end{equation}
for all $\mathbf{w} \in \prod_{K\in \cT_h^{\Omega_I}}\cP_k(K)^3$.
\end{lemma}

\begin{proof}
By definition of $\mathtt{s}_{\cF}$, for any $\mathbf{w} \in \prod_{K\in \cT_h^{\Omega_C}}\cP_k(K)^3$,
\begin{multline}\label{transfer0}
\norm{\mathtt{s}_{\cF}^{1/2} h_{\cF}^{1/2} \mean{\sigma^{-1}\mathbf{w}}_{\cF}}^2_{0,\cF_h^{\Omega_C}} =  
\sum_{F\in \cF_h^{\Omega_C}} h_F \norm{ \mathtt{s}_F^{1/2}\mean{\sigma^{-1}\mathbf{w}}_F }^2_{0,F}\\ 
\leq \sum_{K\in \cT_h^{\Omega_C}} \sum_{F\in \cF(K)} h_F \norm{ \mathtt{s}_F^{1/2}\sigma_K^{-1}\mathbf{w}_K }^2_{0,F}
\leq  \sum_{K\in \cT_h^{\Omega_C}}  h_K 
\norm{ \sigma_K^{-1/2}\mathbf{w}_K }^2_{0,\partial K}.
\end{multline}

Similarly, 
\begin{multline}\label{transfer}
\norm{ \mathtt{s}_{\cE}^{1/2} h_{\cE} \mean{\sigma^{-1} \mathbf{w}}_{\cE} }^2 _{0,\cE_h} =  
\sum_{e \in \cE_h} h_e^2 \norm{\mathtt{s}_{e}^{1/2} \mean{\sigma^{-1} \mathbf{w}}_e }^2_{0,e} \\
\leq \sum_{T\in \cF_h^\Gamma} \sum_{e\in \cE(T)}  h_e^2 \norm{\mathtt{s}_{e}^{1/2} \sigma_{K_T}^{-1} 
\mathbf{w}_{K_T}}^2_{0,e}
\leq \sum_{T\in \cF_h^\Gamma }
  h_{T}^2 \norm{\sigma_{K_T}^{-1/2} \mathbf{w}_{K_T}}^2_{0,\partial T}\, ,
\end{multline} 
where $K_T\in \cT_h^{\Omega_C}$ is such that $T=\partial K_T \cap \Gamma$.
It follows from \eqref{discreteTrace3D} that 
\begin{multline*}
\norm{\mathtt{s}_{\cE}^{1/2} h_{\cE} \mean{\sigma^{-1} \mathbf{w}}_{\cE}}^2 _{0,\cE_h^{\Omega_I}} \leq C^*
\sum_{T \in \cF_h^\Gamma}  h_{T} \norm{\sigma_{K_T}^{-1/2} \mathbf{w}_{K_T}}^2_{0, T} \leq C^*
\sum_{K \in\cT^{\Omega_C}_h}  h_{K} \norm{\sigma_K^{-1/2} \mathbf{w}_K}^2_{0, \partial K} 
\end{multline*} 
and \eqref{discIneq1} follows by applying again the discrete trace inequality \eqref{discreteTrace3D} in the last estimate and in \eqref{transfer0}. 
Finally, for any $\mathbf{w} \in \prod_{K\in \cT_h^{\Omega_I}}\cP_k(K)^3$,
\begin{equation}\label{transfer1}
\norm{h_{\cF}^{1/2} \mean{\mathbf{w}}_{\cF}}^2_{0,\cF_h^{\Omega_I}} = \sum_{F\in \cF_h^{\Omega_I}} h_F \norm{\mean{\mathbf{w}}_F}_{0,F}^2 
\leq  \sum_{K\in \cT_h^{\Omega_I}} h_K \norm{ \mathbf{w}_K }^2_{0,\partial K} 
\end{equation}
and \eqref{discIneq2} follows again from \eqref{discreteTrace3D}.
\end{proof}

\begin{prop}\label{boundedness}
There exists a constant $M>0$ independent of $h$ such that 
\[
| A_h((\bu, \phi, c), (\bv, \varphi, m)) | \leq M \norm{(\bu, \phi, c)}_* \norm{(\bv,\varphi, m)}
\]
for all $(\bu, \phi,c)$, $(\bv,\varphi,m)\in \mathbf{X}^s(\cT_h^{\Omega_C})\times \rmH^{1+s}(\cT_h^{\Omega_I})\times \C$, with $s>1/2$.
\end{prop}
\begin{proof}
By the Cauchy-Schwarz inequality, we have that 
$$ 
\begin{array}{l}
|A_h^{\Omega_C}((\bu, \phi,c), (\bv, \varphi,m))| \\[.1cm]
\qquad \leq \omega \norm{\mu^{1/2}\bu}_{0,\Omega_C} \norm{\mu^{1/2}\bv}_{0,\Omega_C} + 
\norm{\sigma^{-1/2}\curl_h \bu}_{0,\Omega_C} \norm{\sigma^{-1/2}\curl_h \bv}_{0,\Omega_C}\\[.1cm] 
\qquad + \norm{\mathtt{s}_{\cF}^{1/2} h_{\cF}^{1/2} \mean{\sigma^{-1}\curl_h \bu}_{\cF}}_{0, \cF_h^{\Omega_C}}
 \norm{\mathtt{s}_{\cF}^{-1/2} h_{\cF}^{-1/2} \jump{(\bv, \varphi,m)}_{\cF}}_{0, \cF_h^{\Omega_C}}\\[.1cm]
\qquad + \norm{\mathtt{s}_{\cF}^{1/2} h_{\cF}^{1/2} \mean{\sigma^{-1}\curl_h \bv}_{\cF}}_{0, \cF_h^{\Omega_C}}
 \norm{\mathtt{s}_{\cF}^{-1/2} h_{\cF}^{-1/2} \jump{(\bu, \phi,c)}_{\cF}}_{0, \cF_h^{\Omega_C}}\\[.1cm]
 \qquad + \mathtt{a}^{\Omega_C} \norm{\mathtt{s}_{\cF}^{-1/2}h_{\cF}^{-1/2} \jump{(\bu, \phi,c)}_{\cF}}_{0,\cF_h^{\Omega_C}} \norm{\mathtt{s}_{\cF}^{-1/2}h_{\cF}^{-1/2} \jump{(\bv, \varphi,m)}_{\cF}}_{0,\cF_h^{\Omega_C}}.
\end{array}
$$
Applying  \eqref{discIneq1} with $\mathbf{w} = \curl_h \bv$ we obtain 
\[
|A_h^{\Omega_C}((\bu, \phi,c), (\bv, \varphi, m))| \leq (1+ C_\Omega + \mathtt{a}^{\Omega_C})\,  \norm{(\bu, \phi,c)}_* \norm{(\bv,\varphi,m)}
\]
for all $(\bu, \phi,c)$ and $(\bv,\varphi,m)\in \mathbf{X}^s(\cT_h^\Omega)\times \rmH^{1+s}(\cT_h^{\Omega_I}) \times \C $. 
On the other hand,
$$
\begin{array}{l}
|A_h^{\Omega_I}((\bu, \phi, c), (\bv, \varphi,m))| \leq \omega \mu_0 \norm{\nabla_h \phi+c \boldsymbol \rho}_{0,\Omega_I} \norm{\nabla_h \varphi + m \boldsymbol \rho}_{0,\Omega_I} \\[.1cm]
\quad\qquad+ \omega \mu_0 \norm{h_{\cF}^{1/2} \mean{\nabla_h \varphi+m \boldsymbol \rho}_{\cF}}_{0,\cF_h^{\Omega_I}} \norm{h_{\cF}^{-1/2} \jump{\phi\n}_{\cF}}_{0,\cF_h^{\Omega_I}} \\[.1cm]
\quad\qquad+ \omega \mu_0 \norm{h_{\cF}^{1/2} \mean{\nabla_h \phi+c \boldsymbol \rho}_{\cF}}_{0,\cF_h^{\Omega_I}} \norm{h_{\cF}^{-1/2} \jump{\varphi\n}_{\cF}}_{0,\cF_h^{\Omega_I}}\\[.1cm]
\quad\qquad+ \alpha \norm{\mathtt{s}_{\cF}^{-1/2}h_{\cE}^{-1} \jump{\phi\tg}_{\cE}}_{0,\cE_h} \norm{\mathtt{s}_{\cF}^{-1/2}h_{\cE}^{-1} \jump{\varphi\tg}_{\cE}}_{0,\cE_h}\\[.1cm]
\quad\qquad+  \norm{\mathtt{s}_{\cF}^{1/2}h_{\cE}\mean{\sigma^{-1} \curl_h \bv}_{\cE}}_{0,\cE_h} \norm{\mathtt{s}_{\cF}^{-1/2}h_{\cE}^{-1} \jump{\phi\tg}_{\cE}}_{0,\cE_h}\\[.1cm]
\quad\qquad+ \norm{\mathtt{s}_{\cF}^{1/2}h_{\cE}\mean{\sigma^{-1} \curl_h \bu}_{\cE}}_{0,\cE_h} \norm{\mathtt{s}_{\cF}^{-1/2}h_{\cE}^{-1} \jump{\varphi\tg}_{\cE}}_{0,\cE_h}\\[.1cm]
\quad\qquad+\mathtt{a}^{\Omega_I}\norm{h_{\cF}^{-1/2} \jump{\phi\n}_{\cF}}_{0,\cF_h^{\Omega_I}} \norm{h_{\cF}^{-1/2} \jump{\varphi\n}_{\cF}}_{0,\cF_h}
\end{array}
$$
and it follows from  \eqref{discIneq2} (applied with $\mathbf{w}= \nabla_h \varphi+ m \boldsymbol \rho$
) and  \eqref{discIneq1} (applied with $\mathbf{w}= \curl_h \bv$) that
\[
|A_h^{\Omega_I}((\bu, \phi,c), (\bv, \varphi, m))| \leq (1 + C_{\Omega_I} + C_{\Omega}+ \mathtt{a}^{\Omega_I}+ \alpha) \,  \norm{(\bu, \phi,c)}_* \norm{(\bv,\varphi,m)},
\]
which gives the result. 
\end{proof}

\begin{prop}\label{discElip} 
There exists a constant $\alpha_0>0$ independent of the mesh size and the coefficients such that if $\min(\mathtt{a}^{\Omega_C}, \mathtt{a}^{\Omega_I}, \alpha)\geq \alpha_0$
then, 
\begin{equation}\label{elip}
\text{Re}\left[ (1 - \imath) A_h((\bv, \varphi,m), (\overline \bv, \overline \varphi,\overline m))\right] \geq \dfrac{1}{2} \norm{(\bv, \varphi,m)}^2 \qquad \forall (\bv, \varphi,m)\in \mathbf{X}_h\times V_h\times \C.
\end{equation} 
\end{prop}
\begin{proof}
By definition of $A_h(\cdot, \cdot)$,  
\begin{equation}\label{elip0}
\begin{array}{l}
\text{Re}\left[(1 - \imath) A_h((\bv, \varphi,m), (\overline \bv, \overline \varphi,\overline m))\right] = \omega \norm{ \mu^{1/2} \bv}_{0,\Omega_C}^2 + 
 \norm{\sigma^{-1/2} \curl_h \bv}_{0,\Omega_C}^2\\[.1cm] 
\qquad+ 2  \text{Re} \left( \mean{\sigma^{-1} \curl_h \bv}_{\cF}, \jump{(\overline \bv, \overline \varphi,\overline m)}_{\cF} \right)_{\cF_h^{\Omega_C}} 
+  \mathtt{a}^{\Omega_C} \norm{h_{\cF}^{-1/2}\jump{(\bv, \varphi,m)}_{\cF}}_{0,\cF_h^{\Omega_C}}^2\\[.1cm] 
 \qquad +\omega \mu_0 \norm{\nabla_h \varphi+m \boldsymbol \rho}_{0,\Omega_C}^2 - 2 \omega \mu_0 \text{Re} \left( 
 \mean{\nabla_h \varphi+ m \boldsymbol \rho}_{\cF}, \jump{\overline \varphi \n}_{\cF}\right)_{\cF_h^{\Omega_I}} \\[.1cm] 
 \qquad+  \mathtt{a}^{\Omega_I}\norm{h_{\cF}^{-1/2}\jump{\varphi\n}_{\cF}}^2_{0,\cF_h^{\Omega_I}} 
- 2 \text{Re}  \left( \mean{\sigma^{-1} \curl_h \bv}_{\cE}, \jump{\overline \varphi \tg}_{\cE} \right)_{\cE_h} 
+  \alpha \norm{h_{\cE}^{-1}\jump{\varphi\tg}_{\cE}}^2_{0, \cE_h}.
\end{array}
\end{equation}
 
It follows from the Cauchy-Schwarz inequality and \eqref{discIneq1} that,
\begin{equation}\label{elip1}
\begin{array}{l}
2 |  \text{Re} \left( \mean{\sigma^{-1} \curl_h \bv}_{\cF}, \jump{(\overline \bv, \overline \varphi,\overline m)}_{\cF} \right)_{\cF_h^{\Omega_C}}|
\\[.1cm]
\qquad \leq 2 \norm{\mathtt{s}_{\cF}^{1/2} h_{\cF}^{1/2} \mean{\sigma^{-1} \curl_h \bv }_{\cF}}_{0,\cF_h^{\Omega_C}} 
\norm{\mathtt{s}_{\cF}^{-1/2} h_{\cF}^{-1/2} \jump{( \bv,  \varphi,m) }_{\cF}}_{0,\cF_h^{\Omega_C}} \\[.1cm]
\qquad\leq 2C_{\Omega_C} \norm{\sigma^{-1/2}\curl_h \bv}_{0,\Omega_C} \norm{\mathtt{s}_{\cF}^{-1/2} h_{\cF}^{-1/2} \jump{( \bv,  \varphi, m) }_{\cF}}_{0,\cF_h^{\Omega_C}} \\[.1cm]
\qquad\leq \frac{1}{4} \norm{\sigma^{-1/2}\curl_h \bv}_{0,\Omega_C}^2 + 4 C_{\Omega_C}^2   
\norm{\mathtt{s}_{\cF}^{-1/2} h_{\cF}^{-1/2}  \jump{( \bv,  \varphi, m) }_{\cF}}_{0,\cF_h^{\Omega_C}}^2.
\end{array}
\end{equation}
Similarly, by virtue of \eqref{discIneq2},
\begin{equation}\label{elip2}
\begin{array}{l}
2 | \text{Re} \left( 
 \mean{\nabla_h \varphi+m \boldsymbol \rho}_{\cF}, \jump{\overline \varphi \n}_{\cF}\right)_{\cF_h^{\Omega_I}} | \leq
2\norm{h_{\cF}^{1/2} \mean{\nabla_h \varphi+ m \boldsymbol \rho}_{\cF} }_{0,\cF_h^{\Omega_I}}  \norm{h_{\cF}^{-1/2} \jump{\overline \varphi \n}_{\cF} }_{0,\cF_h^{\Omega_I}}\\[.1cm]
\qquad \qquad \leq 2 C_{\Omega_I} \norm{\nabla_h \varphi+m \boldsymbol \rho}_{0,\Omega} \norm{h_{\cF}^{-1/2} \jump{ \varphi \n}_{\cF} }_{0,\cF_h^{\Omega_I}} \\[.1cm]
\qquad\qquad \leq \frac{1}{2} \norm{\nabla_h \varphi + m \boldsymbol \rho}_{0,\Omega}^2 + 4 C_{\Omega_I}^2 \norm{h_{\cF}^{-1/2} \jump{ \varphi \n}_{\cF} }_{0,\cF_h^{\Omega_I}}^2.	
\end{array}
\end{equation}
Finally, using \eqref{discIneq1} we have that 
\begin{equation}\label{elip3}
\begin{array}{l}
2 |\text{Re}  \left( \mean{\sigma^{-1} \curl_h \bv}_{\cE}, \jump{\overline \varphi \tg}_{\cE} \right)_{\cE_h}|\leq 
2 \norm{\mathtt{s}_{\cE}^{1/2} h_{\cE}\mean{\sigma^{-1} \curl_h \bv}_{\cE}}_{0,\cE_h} 
\norm{\mathtt{s}_{\cE}^{-1/2} h_{\cE}^{-1}\jump{ \varphi \tg}_{\cE}}_{0,\cE_h}  \\[.1cm]
\qquad \leq  
2 C_\Gamma  \norm{\sigma^{-1/2} \curl_h \bv}^2_{0,\Omega_C}
\norm{\mathtt{s}_{\cE}^{-1/2} h_{\cE}^{-1}\jump{ \varphi \tg}_{\cE}}_{0,\cE_h}\\[.1cm] 
\qquad\leq  \frac{1}{4}  
\norm{\sigma^{-1/2} \curl_h \bv}_{0,\Omega_C}^2 + 4 C_{\Omega_C}^2 \norm{\mathtt{s}_{\cE}^{-1/2} h_{\cE}^{-1}\jump{ \varphi \tg}_{\cE}}_{0,\cE_h}^2.
\end{array}
\end{equation}
Combining \eqref{elip0} with  \eqref{elip1}-\eqref{elip3} and choosing $\alpha_0 = 1/2 + 4 C_\Omega^2 + 4C_{\Omega_I}^2$ we obtain \eqref{elip}.
\end{proof}

We are now in a position to prove the $\norm{\cdot}$-stability of the DG scheme~\eqref{DG-FEM}. 
\begin{theorem}\label{LM}
Assume that $\sigma^{-1} \bj\in \mathbf{H}^{1/2+s}(\cT_h^{\Omega_C})$ and $\min(\mathtt{a}^\Omega, \mathtt{a}^{\Omega_I}, \alpha)\geq \alpha_0$. Then, 
there exits a unique $(\bH_{h},\psi_h,k_h) \in \mathbf{X}_h\times V_h\times \C$ solution of Problem~\eqref{DG-FEM}.
Moreover if $(\bH, \psi,k)\in [\mathbf{H}(\curl,\Omega)\times \rmH^1(\Omega_I)\times \C] \cap  [\mathbf{X}^s(\cT_h^{\Omega_C}) \times \rmH^{1+s}(\cT_h^{\Omega_I}) \times \C]$ is  the solution to \eqref{ModelProblem1}-\eqref{ModelProblem6}
then 
\begin{equation} \label{Cea}
\norm{(\bH - \bH_{h}, \psi - \psi_h, k -k_h)} \leq (1 + 2\sqrt{2} M) \inf_{(\bv, \varphi)\in \mathbf{X}_h\times V_h} \norm{(\bH - \bv, \psi -\varphi, 0)}_*.
\end{equation}
\end{theorem}

\begin{proof}
The well posedness of Problem~\eqref{DG-FEM}  follows immediately from Proposition~\ref{discElip}.

Moreover we deduce from Proposition \ref{discElip} and  the consistency of the scheme that 
$$
\begin{array}{l}
\frac{1}{2} \norm{(\bH_{h}- \bv, \psi_h - \varphi, k_h-m)}^2 \\[.1cm]
\qquad \leq \text{Re} \left[ (1 - \imath) A_h((\bH_{h}- \bv, \psi_h - \varphi, k_h-m),(\bH_{C,h}- \bv, \psi_h - \varphi, k_h-m)) \right]  \\[.1cm]
\qquad = \text{Re} \left[ (1 - \imath) A_h((\bH- \bv, \psi - \varphi, k-m),(\bH- \bv, \psi - \varphi, k-m)) \right]
\end{array}
$$
for all $(\bv, \varphi,m)\in \mathbf{X}_h\times V_h\times \C$. Then from Proposition \ref{boundedness} we have
\[
\norm{(\bH_{h}- \bv, \psi_h - \varphi, k_h-m)} \leq 2 \sqrt{2} M \norm{(\bH - \bv, \psi - \varphi,k-m)}_\ast.
\]
The result follows now from the triangle inequality. 
\end{proof}

\section{Asymptotic error estimates}\label{sec5}

We denote by $\bPi_{h,m}^{\text{curl}}$ the $m$-order $\mathbf{H}(\curl, \Omega_C)$-conforming N\'ed\'elec interpolation operator of the second kind, see for example \cite{NED86} or \cite[Section 8.2]{Monk}.
It is well known  that $\bPi_{h,m}^{\text{curl}}$ is bounded on 
$\mathbf{H}(\curl, \Omega_C)\cap \mathbf{H}^s(\curl, \cT_h^{\Omega_C})$ for $s>1/2$,
where
\[
\mathbf{H}^s(\curl, \cT_h^{\Omega_C}) := \set{\bv \in \mathbf{H}^s(\cT_h^{\Omega_C});\quad \curl_h \bv \in \mathbf{H}^{s}(\cT_h^{\Omega_C})}.
\]
Moreover, there exists a constant $C_1>0$ independent of $h$ such that (cf. \cite{AVbook})
\begin{equation}\label{errorInterp1}
\norm{\bu - \bPi_{h,m}^{\text{curl}} \bu}_{0, \Omega_C} + \norm{\curl(\bu - \bPi_{h,m}^{\text{curl}}\bu)}_{0, \Omega_C} \leq C_1 h^{\min(s, m)} \big(  \norm{\bu}_{s, \cT_h^{\Omega_C}} + \norm{\curl_h \bu}_{s, \cT_h^{\Omega_C}} \big).
\end{equation}

We introduce 
$
\mathbf{L}^2_t(\Gamma) = \set{\boldsymbol{\varphi}\in \mathbf{L}^2(\Gamma);\,\,  \boldsymbol{\varphi}\cdot \n = 0}
$
and consider the $m$-order order Brezzi-Douglas-Marini (BDM) finite element approximation of the space 
\[
\mathbf{H}(\text{div}_\Gamma, \Gamma) := \set{ \boldsymbol{\varphi}\in \mathbf{L}_t^2(\Gamma); \quad 
\text{div}_\Gamma \boldsymbol{\varphi}\in L^2(\Gamma) }
\]
relatively to the mesh $\cF_h^\Gamma$ (see, e.g. \cite{Boffi}). It is given by 
\[
\mathcal{\mathbf{BDM}}(\cF_h^\Gamma) = \set{\boldsymbol{\varphi}\in  \mathbf{H}(\text{div}_\Gamma, \Gamma); \quad 
\boldsymbol{\varphi}|_T \in \cP_m(T)^2,\quad \forall T\in \cF_h^\Gamma}.
\]
The corresponding interpolation operator $\Pi_{h,m}^{\text{BDM}}$ is bounded on $
\mathbf{H}(\text{div}_\Gamma, \Gamma)\cap \prod_{T\in \cF_h^\Gamma} \rmH^\delta(T)^2$ for all $\delta>0$ and we recall that it is uniquely characterized on each  
$T\in \cF_h^\Gamma$ by the conditions
\begin{equation}\label{BDMfreedom1}
\int_e \Pi_{h,m}^{\text{BDM}} \boldsymbol{\varphi}\cdot \n_T q = \int_e \boldsymbol{\varphi}\cdot \n_T q\quad 
\forall q \in \cP_m(e),\quad \forall e\in \cE(T),
\end{equation}
\begin{equation}\label{BDMfreedom2}
\int_T \Pi_{h,m}^{\text{BDM}} \boldsymbol{\varphi}\cdot  \mathbf{q} = \int_T \boldsymbol{\varphi}\cdot  \mathbf{q}\quad 
\forall \mathbf{q} \in \cP_{m-2}(T)^2 + \mathbf{S}_{m-1}(T),
\end{equation}
where $\mathbf{S}_{m-1}(T):= \set{\mathbf{q}\in \tilde\cP_{m-1}(T)^2;\quad \mathbf{q}\cdot \begin{pmatrix}x_1\\x_2
\end{pmatrix} = 0}$ with $\tilde\cP_{m-1}(T)$ representing the set of homogeneous polynomials of degree $m-1$ and 
$\begin{pmatrix}x_1\\x_2\end{pmatrix}$ being the local variable on the plane containing $T$.

The  commuting diagram property
\begin{align}
\label{commuting2} (\bPi_{h,m}^{\text{curl}} \bu) \times \n_\Gamma &= \Pi_{h,m}^{\text{BDM}} (\bu \times \n_\Gamma)
\end{align}
holds true for all $\bu \in \mathbf{H}(\curl, \Omega_C)\cap \mathbf{H}^s(\curl, \cT_h^{\Omega_C})$, $s>1/2$, see \cite[section 9]{Hiptmair1} for more details. 

For all $K\in \cT_h^{\Omega_I}$ we define the local interpolation operator $\tilde\pi_{K,m}:\rmH^{1+s}(K)\to \tilde\cP_{m}(K)$, $s>1/2$ as follows: recalling the definition of $\tilde\cP_{m}(K)$ given in \eqref{tildePm}
\begin{itemize}
\item if $\partial K \cap \Gamma \not \in \mathcal F_h^\Gamma$ then $\tilde\cP_{m}(K)= \cP_{m}(K)$ and we take $\tilde\pi_{K,m} = \pi_{K,m}$, where $\pi_{K,m}$ is defined as in \cite[Section 5.6]{Monk};
\item if $\partial K \cap \Gamma=T  \in \mathcal F_h^\Gamma$ then $\tilde\cP_{m}(K)=\cP_m(K) + \cP_{m+1}^T(K)$ and $\tilde \pi_{K,m}$ is defined by changing the conditions defining $\pi_{K,m}$ on $T$ and on the edges composing $T$ into 
\begin{equation}\label{freedomFb}
\int_T \tilde\pi_{K,m} p q = \int_T p q \qquad \forall q \in \cP_{m-2}(T)
\end{equation}
and
\begin{equation}\label{freedomEb}
\int_e \tilde\pi_{K,m} p q = \int_e p q \qquad \forall q \in \cP_{m-1}(e), \quad \forall e\in \cE(F)
\end{equation}
respectively. The remaining degrees of freedom are the same as those defining $\pi_{K,m}$, see \cite[Section 5.6]{Monk}.
\end{itemize}
We notice that 
$
\text{dim}(\cP_m(K) + \cP_{m+1}^T(K))= \text{dim}(\cP_m(K)) + m+1
$
and the number of degrees of freedom defining $\tilde \pi_{K,m}$ is equal to the number of degrees of freedom of  $\pi_{K,m}$ 
plus $\text{dim}(\cP_{m-2}(T)) - \text{dim}(\cP_{m-3}(T)) = m-1$ additional degrees of freedom on $T$ and one additional degree of freedom on each of the three edges of $T$, which gives a total of $\text{dim}(\cP_m(K)) + m+1$ degrees of freedom. Using this fact, it is straightforward to show that $\tilde \pi_{K,m}$ is uniquely determined 
on elements $K\in \cT_h^{\Omega_I}$ with a face $T$ lying on $\Gamma$. Moreover, it is clear that the 
corresponding global $\rmH^1(\Omega)$-conforming interpolation operator $\tilde \pi_{h,m}$ satisfies the following  interpolation error estimate. 
\begin{prop}
If $p\in \rmH^1(\Omega_I)\cap \rmH^{1+s}(\cT_h^{\Omega_I})$ with $s>1/2$, there exists a constant $C>0$ independent of $h$ such that 
\begin{equation}\label{interp1}
\norm{\nabla (p - \tilde \pi_{h,m} p)}_{0,\Omega_I} \leq C h^{\min(m, s)} \norm{p}_{1+s, \cT_h^{\Omega_I}}.
\end{equation}
\end{prop}
\begin{proof}
See \cite[Lemma 5.47]{Monk} and \cite[Theorem 5.48]{Monk}.
\end{proof}

The commuting diagram property stated in the next proposition is the reason for which we use 
$\tilde \pi_h$ instead of the usual Lagrange interpolation operator.   
\begin{prop}\label{commuting3}
For any $p\in H^1(\Omega)\cap \rmH^{1+s}(\cT_h^{\Omega_I})$, with $s>1/2$, it holds
\[
\nabla \tilde \pi_{h,m} p \times \n_\Gamma = \Pi_{h,m}^{\text{BDM}} (\nabla p \times \n_\Gamma).
\]
\end{prop}

\begin{proof}
We first notice that $\nabla \tilde \pi_{h,m} p \times \n_\Gamma \in \mathbf{H}(\text{div}_\Gamma, \Gamma)$  and 
$\nabla \tilde \pi_{h,m} p \times \n_\Gamma \in \cP_m(T)$ for all $T\in \cF_h^\Gamma$. Hence, $\nabla \tilde \pi_{h,m} p \times \n_\Gamma \in \mathcal{\mathbf{BDM}}(\cF_h^\Gamma)$. To show that $\nabla \tilde \pi_{h,m} p \times \n_\Gamma = \Pi_{h,m}^{\text{BDM}} (\curl_\Gamma p)$, 
it is sufficient to compare the degrees of freedom of these two tangential fields on each triangle $T\in \cF_h^\Gamma$.
On the one hand, for all $q\in \cP_m(e)$, $e\in \cE(T)$,  
$$
\begin{array}{l}
\displaystyle{\int_e (\nabla \tilde \pi_{h,m} p \times \n_\Gamma - \Pi_{h,m}^{\text{BDM}} (\nabla p \times \n_\Gamma))\cdot \n_F q}\\[.3cm]
\qquad  = \displaystyle{\int_e \nabla \left((\tilde\pi_{h,m} p -  p)\times n_\Gamma \right)\cdot \n_F q 
= \int_e \frac{ \partial(\tilde\pi_{h,m} p -  p)}{\partial\tg_e} q} \\[.3cm]
\qquad=\displaystyle{ - \int_e (\tilde\pi_{h,m} p -  p) \frac{ \partial q}{\partial\tg_e} + 
(\tilde\pi_{h,m} p -  p)(\mathbf{a}_e) q(\mathbf{a}_e) - (\tilde\pi_{h,m} p -  p)(\mathbf{b}_e) q(\mathbf{b}_e)=0}\, ,
\end{array}
$$
where the last identity follows from the fact that $\tilde\pi_{h,m} p$ and  $p$ must 
coincide at the endpoints $\mathbf{a}_e$ and $\mathbf{b}_e$ of edge $e$ (by definition of the $\tilde \pi_{h,m}$) and  
from \eqref{freedomEb}, taking into account that 
$\frac{ \partial q}{\partial\tg_e}\in \cP_{m-1}(e)$.

On the other hand, for any $\mathbf{q}\in \cP_{m-2}(T)^2 + \mathbf{S}_{m-1}(T)$, we have that 
$$
\begin{array}{l}
\displaystyle{\int_T (\nabla \tilde \pi_{h,m} p \times \n_\Gamma- \Pi_{h,m}^{\text{BDM}} (\nabla p \times \n_\Gamma))\cdot \mathbf{q} }\\
\qquad \displaystyle{=\int_T \nabla (\tilde \pi_{h,m}p -  p) \times \n_\Gamma  \cdot \mathbf{q}=-\int_T \nabla (\tilde \pi_{h,m}p -  p) \cdot (\mathbf{q}\times \n_\Gamma) }\\
\qquad =\displaystyle{ \int_T  (\tilde\pi_{h,m} p -  p) \,\text{div}_\Gamma (\mathbf{q}\times \n_\Gamma)
- \sum_{e\in \cE(T)}\int_{e} (\tilde\pi_{h,m} p -  p)\, (\mathbf{q}\times \n_\Gamma)\cdot \boldsymbol \nu_T }\\
\qquad \displaystyle{=\int_T  (\tilde\pi_{h,m} p -  p) \,\text{div}_\Gamma (\mathbf{q}\times \n_\Gamma)
- \sum_{e\in \cE(T)}\int_{e} (\tilde\pi_{h,m} p -  p)\, \mathbf{q} \cdot \tg_e }
= 0
\end{array}
$$
by virtue of \eqref{freedomFb} and \eqref{freedomEb}, since $\text{div}_\Gamma (\mathbf{q} \times \n_\Gamma)\in \cP_{m-2}(F)$ 
and $\mathbf{q}\cdot \tg_e\in \cP_{m-1}(e)$.
\end{proof}


Finally, we consider the $\mathbf{L}^2(\cT^{\Omega_C}_h)$-orthogonal projection  ${\bf P}^k_{\cT_h^{\Omega_C}}$ 
onto $\prod_{K\in \cT^{\Omega_C}_h} \cP_k(K)^3$ and the $\mathbf{L}^2(\cT^{\Omega_I}_h)$-orthogonal projection   ${\bf P}^k_{\cT_h^{\Omega_I}}$ onto $\prod_{K\in \cT^{\Omega_I}_h} \cP_k(K)^3$, $k\geq 0$. We denote indifferently 
by $\bPi^k_K$ the restriction of $\bPi^k_{\cT_h^{\Omega_C}}$ and $\bPi^k_{\cT_h^{\Omega_I}}$ to an element $K$.
\begin{lemma}\label{v}
For all $K\in \cT_h$ and  $\mathbf{w}\in \mathbf{H}^{r}(K)$, $r\geq 1/2$, we have
 \begin{equation}\label{proj}
  h_F\norm{\mathbf{w} - {\bf P}^k_K \mathbf{w}}_{0,\partial F} + h_K^{1/2}\norm{\mathbf{w} - {\bf P}^k_K \mathbf{w}}_{0,\partial K}  + \norm{\mathbf{w} - {\bf P}^k_K \mathbf{w}}_{0,K} 
  \leq C h_K^{\min\{r,k+1\}} \norm{\mathbf{w}}_{r,K},
 \end{equation}
 with a constant $C>0$ independent of $h$. 
\end{lemma}
\begin{proof}
See \cite{DiPietroErn}, Lemma 1.58 and Lemma 1.52.
\end{proof}

We are now in a position to prove the main result of this section. 
\begin{theorem}\label{mainDGFEM}
Let $(\bH, \psi,k)\in \mathbf{H}(\curl, \Omega_C)\times \rmH^1(\Omega_I)\times \C$ and $(\bH_{h}, \psi_h,k_h)\in \mathbf{X}_h\times V_h\times C$ be the solutions to \eqref{ModelProblem1}-\eqref{ModelProblem6} and  \eqref{DG-FEM} respectively.
If $\sigma^{-1} \bj\in \mathbf{H}^{1/2+s}(\cT_h^{\Omega_C})$, $(\bH, \psi)\in \mathbf{X}^s(\cT_h^{\Omega_C})\times \rmH^{1+s}(\cT_h^{\Omega_I})$,  with $s>1/2$, and $\min(\mathtt{a}^{\Omega_C}, \mathtt{a}^{\Omega_I}, \alpha)\geq \alpha_0$, then 
\[
\norm{(\bH - \bH_{h}, \psi - \psi_h, k-k_h)} \leq C h^{\min(s, m)} \Big( 
 \norm{\bH}_{s, \cT_h^\Omega} + \norm{\curl\, \bH}_{1/2 + s, \cT_h^\Omega} + \norm{\psi}_{1+s, \cT_h^{\Omega_I}}
\Big),
\]
where $C>0$ is a constant independent of $h$.	
\end{theorem}
\begin{proof}
Taking $(\bv, \varphi)=(\bPi^{\text{curl}}_{h,m}\bH, \tilde \pi_{h,m} \psi)$ in \eqref{Cea} yields 
\[
\norm{(\bH - \bH_{h}, \psi - \psi_h, k-k_h)} \leq (1 + 2\sqrt{2} M) 
\norm{(\bH - \bPi^{\text{curl}}_{h,m}\bH, \psi - \tilde \pi_{h,m} \psi,0)}_*.
\]
All the jumps terms in the right-hand side of the last inequality are zero since the identities  
\begin{equation}\label{true?}
(\bPi^{\text{curl}}_{h,m}\,\bH)\times \n = \bPi^{\text{BDM}}_{h,m}(\bH \times \n_\Gamma) =  \bPi^{\text{BDM}}_{h,m}((\nabla \psi +k \boldsymbol \rho) \times \n_\Gamma) = \left( \nabla \widetilde{\pi}_{h,m} \psi+ k \boldsymbol \rho \right)\times \n_{\Gamma}
\end{equation}
holds true on $\Gamma$ and we also have that 
$$
\jump{(\psi -\tilde \pi_{h,m} \psi)\n}_{\cF}
=
\jump{(\psi -\tilde \pi_{h,m} \psi) \tg}_{\cE} =0\, ,
$$
by construction. Note that in the last equality of \eqref{true?} we have used the fact that $\boldsymbol \rho$ belongs to $H(\curl; \Omega_I)$ and is a piecewise-linear polynomial. It follows that,
$$
\begin{array}{l}
\norm{(\bH - \bPi^{\text{curl}}_{h,m}\bH, \psi - \tilde \pi_{h,m} \psi,0)}_*^2 \\[.2cm]
\qquad= 
\norm{(\omega\mu)^{1/2}(\bH - \bPi^{\text{curl}}_{h,m}\bH)}^2_{0,\Omega_C} 
+ \norm{\sigma^{-1/2}\curl (\bH - \bPi^{\text{curl}}_{h,m}\bH)}^2_{0,\Omega_C} \\[.2cm]
\qquad+ \omega\mu_0 \norm{\nabla_h (\psi - \tilde \pi_{h,m} \psi)}^2_{0,\Omega_I} +
\norm{\mathtt{s}_{\cF}^{1/2} h_{\cF}^{1/2} \mean{\sigma^{-1}\curl (\bH - \bPi^{\text{curl}}_{h,m}\bH)}_{\cF}}^2_{0,\cF_h^{\Omega_C}}\\[.2cm]
\qquad+ \norm{\mathtt{s}_{\cE}^{1/2} h_{\cE} \mean{\sigma^{-1} \curl (\bH - \bPi^{\text{curl}}_{h,m}\bH)}_{\cE}}^2_{0,\cE_h}
+\norm{h_{\cF}^{1/2} \mean{\nabla (\psi - \tilde \pi_{h,m} \psi)}_{\cF}}^2_{0,\cF_h^{\Omega_I}}.
\end{array}
$$
We deduce from the triangle inequality that, 
\begin{multline*}
\norm{\mathtt{s}_{\cF}^{1/2} h_{\cF}^{1/2} \mean{\sigma^{-1}\curl (\bH - \bPi^{\text{curl}}_{h,m}\bH)}_{\cF}}_{0,\cF_h^{\Omega_C}} = 
\norm{\mathtt{s}_{\cF}^{1/2} h_{\cF}^{1/2} \mean{\sigma^{-1}(\curl\bH - {\bf P}^{m-1}_{\cT_h^{\Omega_C}}\curl\, \bH)}_{\cF}}_{0,\cF_h^{\Omega_C}} \\+ 
\norm{\mathtt{s}_{\cF}^{1/2} h_{\cF}^{1/2} \mean{\sigma^{-1}({\bf P}^{m-1}_{\cT_h^{\Omega_C}}\curl\, \bH - \curl\, \bPi^{\text{curl}}_{h,m}\bH)}_{\cF}}_{0,\cF_h^{\Omega_C}} = A_{\Omega_C} + B_{\Omega_C}\, .
\end{multline*}
Using \eqref{discIneq1} yields 
\begin{multline*}
B_{\Omega_C} \leq C_{\Omega_C} \norm{\sigma^{-1/2} ({\bf P}^{m-1}_{\cT_h^{\Omega_C}}\curl\, \bH - \curl\, \bPi^{\text{curl}}_{h,m}\bH) }_{0,\Omega_C}\\ = 
C_{\Omega_C} \norm{\sigma^{-1/2} {\bf P}^{m-1}_{\cT_h^{\Omega_C}} (\curl\, \bH - \curl\, \bPi^{\text{curl}}_{h,m}\bH) }_{0,\Omega_C} \leq 
C_{\Omega_C} \norm{\sigma^{-1/2}\curl ( \bH - \bPi^{\text{curl}}_{h,m} \bH) }_{0,\Omega_C}
\end{multline*}
and by virtue of \eqref{transfer0} we obtain 
\[
A_{\Omega_C}^2 \leq  \sum_{K\in \cT_h^{\Omega_C}}  h_K 
\norm{ \sigma_K^{-1/2}(\curl\, \bH - {\bf P}^{m-1}_K\curl\, \bH) }^2_{0,\partial K}.
\]
Similarly, we consider the splitting
\begin{multline*}
\norm{\mathtt{s}_{\cE}^{1/2} h_{\cE} \mean{\sigma^{-1}\curl (\bH - \bPi^{\text{curl}}_{h,m}\bH)}_{\cE}}_{0,\cE_h} \leq  
\norm{\mathtt{s}_{\cE}^{1/2} h_{\cE} \mean{\sigma^{-1}(\curl\,\bH - {\bf P}^{m-1}_{\cT_h^{\Omega_C}}\curl\, \bH)}_{\cE}}_{0,\cE_h}\\ + 
\norm{\mathtt{s}_{\cE}^{1/2} h_{\cE} \mean{\sigma^{-1}({\bf P}^{m-1}_{\cT_h^{\Omega_C}}\curl\, \bH - \curl\,\bPi^{\text{curl}}_{h,m}\bH)}_{\cE}}_{0,\cE_h} = A_\Gamma + B_\Gamma\,
\end{multline*}
and use \eqref{discIneq1} to obtain
\[
B_\Gamma \leq C_\Gamma \norm{\sigma^{-1/2} ({\bf P}^{m-1}_{\cT_h^{\Omega_C}}\curl\, \bH - \curl\,\bPi^{\text{curl}}_{h,m}\bH) }_{0,\Omega_C} 
\]
\[
=C_\Gamma \norm{\sigma^{-1/2} {\bf P}^{m-1}_{\cT_h^{\Omega_C}} \left( \curl  ( \bH - \bPi^{\text{curl}}_{h,m} \bH) \right) }_{0,\Omega_C}\leq 
C_\Gamma \norm{\sigma^{-1/2}\curl ( \bH - \bPi^{\text{curl}}_{h,m} \bH) }_{0,\Omega_C}.
\]
Moreover, it follows from  \eqref{transfer} that  
\[
A_\Gamma^2 \leq 
\sum_{T\in \cF_h^\Gamma }
  h_{T}^2 \norm{\sigma_{K_T}^{-1/2} (\curl\,\bH - {\bf P}^{m-1}_K\curl\, \bH)}^2_{0,\partial T}.
\]
Finally,
\begin{multline*}
\norm{h_{\cF}^{1/2} \mean{\nabla (\psi - \tilde \pi_{h,m} \psi)}_{\cF}}_{0,\cF_h^{\Omega_I}} \leq 
\norm{h_{\cF}^{1/2} \mean{\nabla \psi - {\bf P}^m_{\cT_h^{\Omega_I}} \nabla \psi}_{\cF}}_{0,\cF_h^{\Omega_I}}\\ + 
\norm{h_{\cF}^{1/2} \mean{{\bf P}^m_{\cT_h^{\Omega_I}} \nabla \psi - \nabla \tilde \pi_{h,m} \psi)}_{\cF}}_{0,\cF_h^{\Omega_I}} = A_{\Omega_I} + B_{\Omega_I}
\end{multline*}
and we derive from \eqref{discIneq2} and \eqref{transfer1} the following estimates
\[
B_{\Omega_I} \leq C_{\Omega_I} \norm{{\bf P}^m_{\cT_h^{\Omega_I}} \nabla \psi - \nabla \tilde \pi_{h,m} \psi}_{0,\Omega_I}\\ \leq 
C_{\Omega_I} \norm{\nabla (\psi - \tilde \pi_{h,m} \psi)}_{0,\Omega_I},
\]
\[
A_{\Omega_I}^2 \leq \sum_{K\in \cT_h^{\Omega_I}} h_K \norm{\nabla \psi - {\bf P}^m_K \nabla \psi }^2_{0,\partial K}.
\]
Combining the last inequalities we deduce that 
\begin{multline*}
\norm{(\bH - \bPi^{\text{curl}}_{h,m}\bH, \psi - \tilde \pi_{h,m} \psi)}_*^2 \leq C \Big(
\norm{\bH - \bPi^{\text{curl}}_{h,m}\bH}^2_{0,\Omega_C} + \norm{\curl (\bH - \bPi^{\text{curl}}_{h,m}\bH)}^2_{0,\Omega_C} \\+  \norm{\nabla_h (\psi - \tilde \pi_{h,m} \psi)}^2_{0,\Omega_I} + \sum_{K\in \cT_h^{\Omega_C}}  h_K 
\norm{ \curl\bH - {\bf P}^{m-1}_K\curl \bH }^2_{0,\partial K} \\+ \sum_{T\in \cF_h^\Gamma }
  h_{T}^2 \norm{\curl\bH - {\bf P}^{m-1}_K\curl \bH}^2_{0,\partial T}+ \sum_{K\in \cT_h^{\Omega_I}} h_K \norm{\nabla \psi - {\bf P}^m_K \nabla \psi }^2_{0,\partial K}
\Big)
\end{multline*}
with $C>0$ independent of $h$. Applying the interpolation error estimates given by \eqref{errorInterp1}, \eqref{interp1} and \eqref{proj} we obtain 
\begin{multline*}
\norm{(\bH - \bPi^{\text{curl}}_{h,m}\bH, \psi - \tilde \pi_{h,m} \psi)}_* \leq C \Big( h^{\min(s, m)}  (
\norm{\bH}_{s, \cT_h^{\Omega_C}} + \norm{\curl \bH}_{s, \cT_h^{\Omega_C}}) + 
h^{\min(s, m)} \norm{\psi}_{1+s, \cT_h^{\Omega_I}}\\  + h^{\min(1/2+s, m)} 
\norm{\curl \bH}_{1/2+s, \cT_h^{\Omega_C}} + h^{\min(s, m+1)}  \norm{\nabla \psi}_{s, \cT_h^{\Omega_I}}
\Big)
\end{multline*}
and the result follows.
\end{proof}

\bibliographystyle{siam}

\begin{thebibliography}{10}

\bibitem{ABGV13}
{\sc A.~Alonso~Rodr\'{\i}guez, E.~Bertolazzi, R.~Ghiloni, and A.~Valli}, {\em
  Construction of a finite element basis of the first de {R}ham cohomology
  group and numerical solution of 3{D} magnetostatic problems}, SIAM J. Numer.
  Anal., 51 (2013), pp.~2380--2402.

\bibitem{AMV}
{\sc A.~Alonso~Rodr{\'{\i}}guez, S.~Meddahi, and A.~Valli}, {\em Coupling
  {DG}-{FEM} and {BEM} for time harmonic eddy current problem},
  arXiv:1611.08127 [math.NA],  (2016).

\bibitem{AV09}
{\sc A.~Alonso~Rodr{\'{\i}}guez and A.~Valli}, {\em A {FEM}-{BEM} approach for
  electro-magnetostatics and time-harmonic eddy-current problems}, Appl. Numer.
  Math., 59 (2009), pp.~2036--2049.

\bibitem{AVbook}
\leavevmode\vrule height 2pt depth -1.6pt width 23pt, {\em Eddy {C}urrent
  {A}pproximation of {M}axwell {E}quations}, Springer-Verlag Italia, Milan,
  2010.

\bibitem{arnoldIP}
{\sc D.~N. Arnold}, {\em An interior penalty finite element method with
  discontinuous elements}, SIAM J. Numer. Anal., 19 (1982), pp.~742--760.

\bibitem{ABP09}
{\sc S.~Au{\ss}erhofer, O.~B{\'{\i}}r{\'o}, and K.~Preis}, {\em Discontinuous
  {G}alerkin formulation for eddy-current problems}, COMPEL, 28 (2009),
  pp.~1081--1090.

\bibitem{BRS02}
{\sc A.~Berm{\'u}dez, R.~Rodr{\'{\i}}guez, and P.~Salgado}, {\em A finite
  element method with {L}agrange multipliers for low-frequency harmonic
  {M}axwell equations}, SIAM J. Numer. Anal., 40 (2002), pp.~1823--1849.

\bibitem{Boffi}
{\sc D.~Boffi, F.~Brezzi, and M.~Fortin}, {\em Mixed finite element methods and
  applications}, vol.~44 of Springer Series in Computational Mathematics,
  Springer, Heidelberg, 2013.

\bibitem{DiPietroErn}
{\sc D.~A. Di~Pietro and A.~Ern}, {\em {M}athematical {A}spects of
  {D}iscontinuous {G}alerkin {M}ethods}, vol.~69 of Math\'ematiques \&
  Applications (Berlin) [Mathematics \& Applications], Springer, Heidelberg,
  2012.

\bibitem{Hiptmair}
{\sc R.~Hiptmair}, {\em Symmetric coupling for eddy current problems}, SIAM J.
  Numer. Anal., 40 (2002), pp.~41--65.

\bibitem{Hiptmair1}
\leavevmode\vrule height 2pt depth -1.6pt width 23pt, {\em Coupling of finite
  elements and boundary elements in electromagnetic scattering}, SIAM J. Numer.
  Anal., 41 (2003), pp.~919--944.

\bibitem{HPSS05}
{\sc P.~Houston, I.~Perugia, A.~Schneebeli, and D.~Sch{\"o}tzau}, {\em Interior
  penalty method for the indefinite time-harmonic {M}axwell equations}, Numer.
  Math., 100 (2005), pp.~485--518.

\bibitem{McLean}
{\sc W.~McLean}, {\em {S}trongly {E}lliptic {S}ystems and {B}oundary {I}ntegral
  {E}quations}, Cambridge University Press, Cambridge, 2000.

\bibitem{MS03}
{\sc S.~Meddahi and V.~Selgas}, {\em A mixed-{FEM} and {BEM} coupling for a
  three-dimensional eddy current problem}, M2AN Math. Model. Numer. Anal., 37
  (2003), pp.~291--318.

\bibitem{Monk}
{\sc P.~Monk}, {\em Finite {E}lement {M}ethods for {M}axwell's {E}quations},
  Oxford University Press, Oxford, 2003.

\bibitem{NED86}
{\sc J.-C. N{\'e}d{\'e}lec}, {\em A new family of mixed finite elements in
  {${\bf R}^3$}}, Numer. Math., 50 (1986), pp.~57--81.

\bibitem{PS03}
{\sc I.~Perugia and D.~Sch{\"o}tzau}, {\em The {$hp$}-local discontinuous
  {G}alerkin method for low-frequency time-harmonic {M}axwell equations}, Math.
  Comp., 72 (2003), pp.~1179--1214.

\end{thebibliography}

\end{document}